\documentclass[12pt,a4paper]{amsart}
\pdfoutput=1
\usepackage[utf8]{inputenc}

\usepackage{hyperref} 
\usepackage{enumerate}
\usepackage{graphicx}

\usepackage{amssymb,amsmath}

\usepackage{epsfig}  		
\usepackage{epic,eepic}       
\usepackage{datetime}

\let\epsilon=\varepsilon
\newcommand*{\Cset}{\mathbb{C}}  

\newcommand*{\inte}{\mathrm{Int}}

\newcommand*{\der}{\mathrm{d}}

\newcommand*{\abs}[1]{\left\lvert#1\right\rvert}   
\newcommand*{\norm}[1]{\left\lVert#1\right\rVert}  
\newcommand*{\sis}[1]{\left\langle#1\right\rangle}  

\newcommand{\kaantyvat}{\operatorname{Mon}(\C^k,\C^m)}
\newcommand{\kaikki}{\operatorname{Lin}(\C^k,\C^m)}
\newcommand{\C}{\mathbb C}
\newcommand{\R}{\mathbb R}

\makeatletter
\def\blfootnote{\gdef\@thefnmark{}\@footnotetext}
\makeatother

\newtheorem{thm}{Theorem}[section]

\newtheorem{lem}[thm]{Lemma}

\theoremstyle{definition}
\newtheorem{definition}[thm]{Definition}

\theoremstyle{remark}

\newtheorem{remark}[thm]{Remark}

\numberwithin{equation}{section}

\begin{document}


\title[Geodesic ray transform with matrix weights]{Geodesic ray transform with matrix weights for piecewise constant functions}



\author{Joonas Ilmavirta}
\author{Jesse Railo}
\thanks{University of Jyväskylä, Department of Mathematics and Statistics}
\address{Department of Mathematics and Statistics\\
  University of Jyv\"askyl\"a\\
  P.O. Box 35 (MaD) FI-40014 University of Jyv\"askyl\"a \\
	Finland}
\email{joonas.ilmavirta@jyu.fi}
\email{jesse.t.railo@jyu.fi}

\date{\today}






\begin{abstract}
We show injectivity of the geodesic X-ray transform on piecewise constant functions when the transform is weighted by a continuous matrix weight. The manifold is assumed to be compact and nontrapping of any dimension, and in dimension three and higher we assume a foliation condition. We make no assumption regarding conjugate points or differentiability of the weight. This extends recent results for unweighted transforms.
\end{abstract}

\keywords{Geodesic ray transform, matrix weight, integral geometry, inverse problems}

\subjclass[2010]{44A12, 65R32, 53A99}


\maketitle

\section{Introduction}

This article studies the weighted geodesic X-ray transform with injective matrix weights on nontrapping Riemannian manifolds with strictly convex boundary.
This operator arises in many applications, and one of the basic questions is if the weighted line integrals over all maximal geodesics determine an unknown function.
We show an injectivity result for a class of piecewise constant functions under the assumptions that the manifold admits a strictly convex function and the weight depends continuously on its coordinates on the unit sphere bundle.
In two dimensions, the result follows for nontrapping manifolds with strictly convex boundary.

Let $(M,g)$ be a compact nontrapping Riemannian manifold with strictly convex boundary. We say that the boundary $\partial M$ is \textit{strictly convex} if its second fundamental form is positive definite at any $x \in \partial M$. A smooth function $f\colon M \to \R$ is said to be \textit{strictly convex} if its Hessian $\nabla_x^2 f\colon T_xM \times T_xM \to \R$ is positive definite at any $x \in M$.
We denote by~$SM$ the unit sphere bundle and by~$\Gamma$ the set of maximal unit speed geodesics.
We say that~$M$ is \textit{nontrapping} if every geodesic in~$\Gamma$ has finite length, and we make this assumption. We denote the unique unit speed geodesic through $(x,v) \in SM$ by $\gamma_{x,v}$, that is, $\gamma_{x,v}(0) = x$ and $\dot{\gamma}_{x,v}(0) = v$.

The \textit{geodesic X-ray transform with matrix weights} is defined as follows.
Fix some integers $m,k\geq1$ and denote the set of linear injections (monomorphisms) $\C^k\to\C^m$ by $\kaantyvat$, which is a subset of the space $\kaikki$ of all linear maps.
If $k=m$, we have $\kaantyvat=GL(\C,k)$, and for $m < k$ we have $\kaantyvat=\emptyset$.
The geodesic X-ray transform with weight $W\in C(SM,\kaikki)$ is defined so that it maps a function $f\colon M\to\C^k$ to $I_Wf\colon\Gamma\to\C^m$ defined by
\begin{equation}
I_Wf(\gamma)
=
\int_0^\tau W(\gamma(t),\dot{\gamma}(t))f(\gamma(t))\der t
\end{equation}
for any maximal geodesic $\gamma\colon[0,\tau]\to M$ whenever the integral is defined.

Injectivity of~$I_W$ for smooth functions was established by Paternain, Salo, Uhlmann, and Zhou~\cite{PSUZ16} if $\dim(M) \geq 3$, $(M,g)$ admits a smooth strictly convex function, and $W \in C^\infty(SM;GL(k,\Cset))$.
The result in~\cite{PSUZ16} is based on the methods developed in the work of Uhlmann and Vasy \cite{UV16}.
In this paper we consider a special case of the matrix weighted X-ray transform for the piecewise constant vector-valued functions.
We gain more flexibility on geometrical assumptions and the proof is considerably simpler, but at the expense of only having the result for a restricted class of functions.
Injectivity was shown recently in the case of piecewise constant functions without weights by Ilmavirta, Lehtonen, and Salo~\cite{ILS17}, and reconstruction was studied in~\cite{L19}.
Our main theorem is the following.

\begin{thm}
\label{thm:thm1}
Let $(M,g)$ be a compact nontrapping Riemannian manifold with strictly convex smooth boundary and $W \in C(SM;\kaantyvat)$.
Suppose that either
\begin{enumerate}[(a)]
\item $\dim(M) = 2$, or
\item $\dim(M) \geq 3$ and $(M,g)$ admits a smooth strictly convex function.
\end{enumerate}
If $f\colon M \to \C^k$ is a piecewise constant function and $I_Wf = 0$, then $f \equiv 0$.
\end{thm}

There is also a local version of theorem~\ref{thm:thm1}, see theorem~\ref{thm:supthm}.

\begin{remark}
The result can be generalized by replacing $\C^k$ and $\C^m$ with two Banach spaces and letting~$W$ be an invertible linear map depending continuously on the coordinates on the sphere bundle~$SM$.
\end{remark}

\begin{remark}
The functions $f\colon M\to\C^k$ are vector-valued in the sense that they are sections of the trivial bundle $M\times\C^k$.
We do not study geodesic X-ray tomography of vector fields or higher order tensor fields. 
\end{remark}

Theorem~\ref{thm:thm1} generalizes results of~\cite{ILS17} for the matrix weighted X-ray transform similar to the one studied in~\cite{PSUZ16}.
Our theorem holds if $\dim(M) \geq 2$, $W \in C(SM;GL(k,\Cset))$, and functions are piecewise constant in comparison to \cite{PSUZ16} where it is assumed that $\dim(M) \geq 3$, $W \in C^\infty(SM; GL(k,\Cset))$ and functions are smooth.
In the Euclidean space $\R^n$ with $n\geq3$ injectivity is known for $C^{1,\alpha}$ weights~\cite{I16} but there is an example of non-injectivity for $W\in C^\alpha$ by Goncharov and Novikov~\cite{GN17}.
Boman constructed an example of a smooth nonvanishing weight on the plane for which the weighted X-ray transform for smooth functions is non-injective~\cite{B93}.
Theorem~\ref{thm:thm1} shows that there are no such counterexamples for piecewise constant functions.
The known results --- including the new ones obtained here --- are summarized on Table~\ref{ourtable}.

We will prove Theorem~\ref{thm:thm1} in Section~\ref{sec:proof}.
We remark that Theorem~\ref{thm:thm1} is based on a generalization of~\cite[Lemma 4.2]{ILS17} whereas the rest of the proof is almost identical to the one in~\cite{ILS17}.
The method in~\cite{ILS17} relies on existence of a strictly convex foliation as in the works of Stefanov, Uhlmann, and Vasy~\cite{UV16,SUV17}, but the method of proof is far simpler.
For a further discussion on the foliation condition see~\cite{PSUZ16} and references therein. We say that $M$ satisfies the \textit{foliation condition} if $M$ admits a strictly convex function. We define the precise meaning of a strictly convex foliation in Section \ref{sec:sketch}.

The matrix weighted X-ray transform is related to recovering a matrix valued connection from its parallel transport \cite{FU01,N02a,PSU12,GPSU16,PSUZ16}.
It also has applications in polarization tomography \cite{S94,NS07,H13} and quantum state tomography~\cite{I16}.

One source of weights is pseudolinearization, a procedure where a nonlinear problem is reduced to a linear problem with weights depending on the unknown.
For a more detailed description of the idea, first appearing in \cite{SU98,SUV16}, see e.g. \cite[Section 8]{IM18}.
Pseudolinearization also leads to an iterative inversion algorithm \cite{I16, SU98, SUV16}.

A boundary reconstruction of the normal derivatives of a function from the broken ray transform reduces to a certain weighted geodesic ray transform on the boundary~\cite{I14}.
Some weights can be realized as attenuation, but we make no such assumptions on~$W$.
The attenuated X-ray transform (see e.g. \cite{ABK98, N02b, SU11, AMU18}) is a well-known special case of the matrix weighted X-ray transform and it is the mathematical basis for the medical imaging method SPECT (see e.g. the survey~\cite{F03}).

\begin{table}[]
\label{ourtable}
\begin{tabular}{lclll}
\textbf{Regularity} & \textbf{Dimension} & {$W = \text{Id}$} & {\textbf{$W \in C^\infty$}} & {\textbf{$W \in C$}} \\
\hline
PWC & $=2$ & {Yes. \cite{ILS17}} & Yes! & Yes! \\
$L^2$ or $C^\infty$ & $=2$ & Unknown. & No. \cite{B93} & No. \cite{B93}  \\
PWC & $\geq3$ & Yes. \cite{UV16}  & Yes. \cite{PSUZ16} & Yes! \\
$L^2$ or $C^\infty$ & $\geq3$ & Yes. \cite{UV16}  & Yes. \cite{PSUZ16} & No. \cite{GN17} 
\end{tabular}
\vspace{.5em}
\caption{Is the X-ray transform injective on manifolds that admit a strictly convex function? The answers in various different cases are summarized below. Here ``PWC'' stands for piecewise constant. A ``Yes!'' with an exclamation mark is a new result proven here. In two dimensions injectivity is known on simple manifolds, but the foliation condition does not imply simplicity.}
\end{table}

\subsection*{Acknowledgements}
J.I.\ was supported by the Academy of Finland (decision 295853).
J.R.\ was supported by the Academy of Finland (Centre of Excellence in Inverse Problems Research at the University of Jyv\"askyl\"a in 2017, Centre of Excellence in Inverse Modelling and Imaging at the University of Helsinki in 2018).
The authors are grateful to Jere Lehtonen and Mikko Salo for helpful discussions related to this work.

\section{Proof}
\label{sec:proof}

\subsection{Definitions}

We follow the notation of~\cite{ILS17}, and any details omitted here can be found there.
We review the main concepts here in a somewhat informal manner.

The \textit{standard $m$-dimensional simplex} is the convex hull of the standard base of $\R^{m+1}$. A \textit{regular $m$-simplex on a manifold~$M$} is a $C^1$-smoothly embedded standard $m$-dimensional simplex. 
The boundary of a regular $m$-simplex is a union of $m+1$ regular $(m-1)$-simplices.

We define the \emph{depth} of a point $x$ in a regular $m$-simplex as follows.
We say that $x$ has depth $0$ if $x$ belongs to the interior of the simplex.
We say that $x$ has depth $1$ if $x$ belongs to the interior of a boundary simplex of the simplex.
Other depths are defined similarly up to depth $m$ at the $m+1$ corner points of the original simplex.

If $\Delta_1$ and $\Delta_2$ are two regular $m$-simplices, we say that their boundaries align nicely if $x \in \Delta_1 \cap \Delta_2$ implies that $x$ has the same depth in both, $\Delta_1$ and $\Delta_2$.

We denote $n=\dim(M)$.
A \textit{regular tiling of a manifold} is a collection of regular $n$-simplices which
cover the manifold,
whose interiors are disjoint, and
whose boundaries align nicely.
An example is given in Figure~\ref{fig:kuva}.
A piecewise constant function is such that the values are constant in the interior of every simplex and zero on their boundaries.
The geometry of corners of simplices is important for our argument, and we review the crucial definitions in more depth.

\begin{figure}
\includegraphics[width=\linewidth]{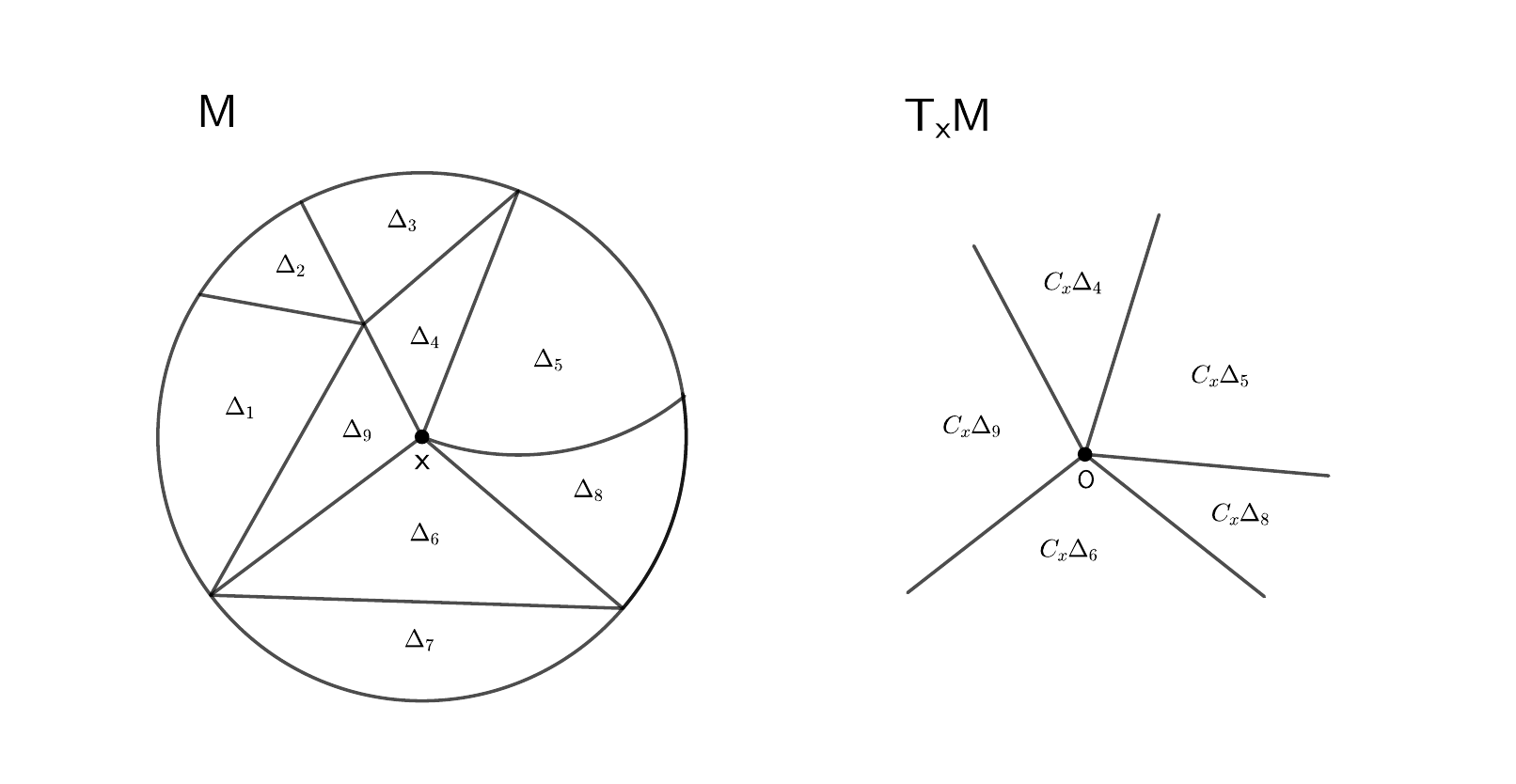}
  \caption{An example of a regular tiling and the tangent space at the point~$x$. The simplices touching~$x$ have their corresponding tangent cones on~$T_xM$. See Definition~\ref{def:C}.}
  \label{fig:kuva}
\end{figure}

\begin{definition}[Tangent cone]
\label{def:C}
Let~$\Delta$ be a regular $m$-simplex in~$M$ with $0\leq m \leq n = \dim(M)$, and let $x \in \Delta$.
Let $\Gamma = \Gamma(x,\Delta)$ be the set of all $C^1$-curves starting at~$x$ and staying in~$\Delta$.
The \textit{tangent cone} of~$\Delta$ at~$x$ is the set
\begin{equation}
C_x\Delta := \{\,\dot{\gamma}(0)\,|\, \gamma \in \Gamma\,\}.
\end{equation}
\end{definition}

\begin{definition}[Tangent function]
Let $f\colon M\to\C^k$ be a piecewise constant function and $x \in M$ with respect to a regular tiling.
Let $\Delta_1,\dots,\Delta_N$ be the simplices of the regular tiling that contain~$x$.
Denote by $v_1,\dots,v_N \in \Cset^k$ the constant values of $f$ in the interior of these simplices.
The \textit{tangent function} $T_xf\colon T_xM \to \Cset^k$ of $f$ at $x$ is defined so that for each $i \in \{1,\dots,N\}$ the function $T_xf$ takes the constant value~$v_i$ in the interior of the tangent cone~$C_x\Delta_i$.
The tangent function takes the value zero in $T_xM \setminus \bigcup_{i=1}^N \inte(C_x \Delta_i)$.
\end{definition}

We stress that the tangent function is not a derivative, as a piecewise constant function is typically not differentiable at the points of interest.
Instead of linearizing the function, we linearize the geometry of the simplices and keep the constant values of the function.

\subsection{Lemmas}

In this subsection we recall a key lemma proved in~\cite[Section 4]{ILS17} and use it to prove a new lemma.

Let~$M$ be a $C^2$-smooth Riemannian surface with $C^2$-boundary.
Suppose the boundary~$\partial M$ is strictly convex at $x \in \partial M$.
Let $\gamma_i, i =1,2$, be two unit speed $C^1$-curves in~$M$ starting nontangentially at~$x$ so that $\abs{\sis{\dot{\gamma}_1(0),\dot{\gamma}_2(0)}} < 1$.

Let the radius $r > 0$ be small enough such that the geodesic ball $B(x,r) \subset M$ is split by the curves $\gamma_i, i =1,2$, into three parts.
Let~$A$ be the middle one. Let $\sigma_i, i =1,2$, be the curves on $T_xM$ with constant speed~$\dot{\gamma}_i(0)$ respectively.
Let~$S$ be the sector in~$T_xM$ laying between~$\sigma_1$ and~$\sigma_2$.

If $h > 0$ and if $v \in T_xM$ is an inward pointing unit vector, let the geodesic~$\gamma_v^h$ be constructed as follows:
Take a unit vector~$w$ normal to~$v$ at~$x$ --- which is unique up to sign --- and let~$w(h)$ be the parallel transport (with respect to the Levi--Civita connection) of $w$ along the geodesic~$\gamma_{x,v}$ by distance~$h$.
Let~$\gamma_v^h$ be the maximal geodesic in the direction of~$w(h)$ at~$\gamma_{x,v}(h)$.
One could denote $\gamma_v^h = \gamma_{\gamma_{x,v}(h),w(h)}$ to be more precise, but we have chosen to keep the notation lighter.

We denote by~$\sigma_v^h$ the corresponding line $hv+w\R$ in~$T_xM$.
The correspondence is not by the exponential map~$\exp_x$ as typically $\gamma_v^h\neq\exp_x(\sigma_v^h)$, but in the sense of lemma~\ref{limitlemma} below. We denote by $\nu \in \partial(SM)$ the inward pointing unit vector of~$\partial M$ at~$x$.

We restate a lemma from~\cite{ILS17} for convenience.

\begin{lem}[{\cite[Lemma 4.1]{ILS17}}]
\label{limitlemma}
Let~$M$ be a $C^2$-smooth Riemannian surface with $C^2$-boundary, which is strictly convex at $x \in \partial M$. There exists an open neighborhood~$U$ of~$\nu$ such that for every $v \in U$ we have
\begin{equation}
\lim_{h \to 0}\frac{1}{h} \int_{\gamma_v^h \cap A} ds = \int_{\sigma_v^1 \cap S} ds.
\end{equation}
\end{lem}

We prove the global result of Theorem~\ref{thm:thm1} by way of proving a local version near a boundary point.
The relevant local version is given below in Lemma~\ref{lem:kulmat}.
A crucial step in its proof is Lemma~\ref{limitlemma}, which allows conversion of the local problem on the manifold into a problem on the tangent space.
However, Lemma~\ref{limitlemma} as stated is not sufficient in the weighted situation, but is used to prove the weighted analogue in Lemma~\ref{keylemma2} below.
Lemma~\ref{keylemma2} can be seen as a generalization of~\cite[Lemma 4.2]{ILS17} for the class of piecewise constant vector-valued functions with matrix weighted integrals.

\begin{lem}
\label{keylemma2}
Let~$M$ be a $C^2$-smooth Riemannian surface with~$C^2$ boundary, which is strictly convex at $x \in \partial M$.
Let~$\tilde{M}$ be such an extension of~$M$ that $x \in \inte(\tilde{M})$.
Let $\Delta\subset M$ be a regular $2$-simplex so that $C_x\Delta\cap T_x\partial M=\{0\}$.
Let $W \in C(\tilde{M};\kaikki)$ and $f\colon\tilde{M} \to \Cset^k$ be a piecewise constant function supported in~$\Delta$.
Then there exists an open neighborhood~$U$ of~$\nu$ such that for every $v \in U$ we have
\begin{equation}
\label{eq:lemma2kaava}
\lim_{h \to 0}\frac{1}{h}\int_{\gamma_v^h}W(\gamma_v^h(s),\dot{\gamma}_v^h(s))f(\gamma_v^h(s))\der s
=
W(x,v^\bot)
\int_{\sigma_v^1}T_xf(\sigma^1_v(s))\der s.
\end{equation}
\end{lem}

\begin{proof}
By linearity we can assume that~$f$ is constant in~$\Delta$.
A piecewise constant function is a linear combination of characteristic functions of interiors of simplices.

Fix $v \in U$ given by Lemma~\ref{limitlemma}.
Let $s_h\in\gamma_v^h$ be any maximizer of
\begin{equation}
s
\mapsto
\norm{W(\gamma_v^h(s),\dot{\gamma}_v^h(s)) - W(x,v^\bot)},
\end{equation}
where --- as throughout this proof --- we use the operator norm of matrices.
We have
\begin{equation}
\sup_{s \in \gamma_v^h}d((\gamma_v^h(s),\dot{\gamma}_v^h(s)),(x,v^\bot))
\to 0
\end{equation}
as $h \to 0$, and so $(\gamma_v^h(s_h),\dot{\gamma}_v^h(s_h)) \to (x,v^\bot)$ as $h \to 0$.
We have $\lim_{h\to 0}\frac{l(\gamma_v^h\cap\Delta)}{h}=\abs{\sigma^1_v}$, and in particular the fraction $\frac{l(\gamma_v^h\cap\Delta)}{h}$ is uniformly bounded for all small $h>0$.

We are ready to compare the weighted integral on the left-hand side of~\eqref{eq:lemma2kaava} to the corresponding integral with the weight frozen to its limit value $W(x,v^\perp)$ (as $h\to0$).
Straightforward estimates give
\begin{equation}
\label{eq:est}
\begin{split}
&\norm{\frac{1}{h}\int_{\gamma_v^h}(W(\gamma_v^h(s),\dot{\gamma}_v^h(s)) - W(x,v^\bot)fds} \\ 
\leq & \frac{l(\gamma_v^h\cap\Delta)}{\abs{h}} \sup_{s \in \gamma_v^h}\norm{f(s)} \sup_{s \in \gamma_v^h}\norm{W(\gamma_v^h(s),\dot{\gamma}_v^h(s)) - W(x,v^\bot)}.
\end{split}
\end{equation}
As~$f$ is bounded, the quotient $\frac{l(\gamma_v^h\cap\Delta)}{h}$ is bounded, and the matrix norm tends to zero as $h\to0$, the left-hand side of~\eqref{eq:est} tends to zero as well.

We may thus conclude that the limit on the left-hand side of~\eqref{eq:lemma2kaava} is the same as
\begin{equation}
W(x,v^\bot)
\lim_{h \to 0}\frac{1}{h}\int_{\gamma_v^h}f(\gamma_v^h(s))\der s.
\end{equation}
That is, the weight can be frozen to its limiting value.
The function $f$ is constant, so up to that constant the integral is just the length of the geodesic segment in~$\Delta$.
Lemma~\ref{limitlemma} shows that the characteristic function of a simplex satisfies~\eqref{eq:lemma2kaava} in the absence of weight.
This concludes the proof.
\end{proof}

We next consider manifolds of dimension $n \geq 2$. Suppose~$\Sigma$ is a hypersurface containing the point $x \in \inte(M)$ and~$\Sigma$ is strictly convex in a neighbourhood of~$x$.
Let~$V$ be a small neighbourhood of~$x$ such that $V \setminus \Sigma$ consists of two open sets which are denoted by~$V_+$ and~$V_-$.
We choose~$V_+$ to be the one for which the boundary section $\partial V_+ \cap \Sigma$ is strictly convex.
Next we state Lemma~\ref{lem:kulmat} that allows one to build a layer stripping argument that is used to prove Theorem~\ref{thm:thm1}.

\begin{lem}
\label{lem:kulmat}
Let~$M$ be a $C^2$-smooth Riemannian manifold, $W \in C(M; \kaantyvat)$ and $f\colon M \to \Cset^k$ be a piecewise constant function.
Fix $x \in \inte(M)$ and let~$\Sigma$ be an $(n-1)$-dimensional hypersurface through~$x$.
Suppose that~$V$ is a neighbourhood of~$x$ so that
\begin{itemize}
\item $V$ intersects only simplices containing~$x$,
\item $\Sigma$ is strictly convex in~$V$,
\item $f|_{V_-} = 0$, and
\item $I_W f = 0$ over every maximal geodesic in~$V$ having endpoints on~$\Sigma$.
\end{itemize}
Then $f|_V = 0$.
\end{lem}

\begin{proof}
The lemma follows from Lemma~\ref{keylemma2} using ideas developed in \cite[Lemma 5.1 and Lemma 6.2]{ILS17}.
We summarize the idea briefly, details can be found in the cited paper.

Consider two dimensions first.
Take a unit vector $v\in T_xM$ pointing towards~$V_+$ and $h>0$.
Define the geodesic~$\gamma_v^h$ as above.
The weighted integrals of $f$ over these geodesics vanish by assumption.
By Lemma~\ref{keylemma2} and injectivity of~$W$ everywhere on the sphere bundle, we find that~$T_xf$ integrates to zero over~$\sigma^1_v$.

This argument reduces the X-ray tomography problem on~$M$ to the corresponding problem on~$T_xM$.
We have the freedom to vary the direction~$v$, and any open set is sufficient.
The Euclidean problem is unweighted and can be solved by explicit calculation, see~\cite[Lemma 3.1]{ILS17}.
The calculation is based on describing the direction~$v$ by a parameter and computing derivatives of high orders with respect to that parameter.
That these derivatives uniquely determine the values of~$T_xf$ in the cones, boils down to the invertibility of a Vandermonde matrix.

In higher dimensions one can proceed as follows.
Take a unit tangent vector $w$ tangential to~$\Sigma$ and a unit vector $v$ pointing towards~$V_+$.
As above, we can define the geodesics~$\gamma_{v,w}^h(t)$, where we now keep the dependence on~$w$ explicit.
Near the point~$x$ the function $(h,t)\mapsto\gamma_{v,w}^h(t)$ defines a smooth two-dimensional submanifold $S_{v,w}\subset M$.
Now the geodesics~$\gamma_{v,w}^h(t)$ are geodesics on both~$M$ and~$S_{v,w}$ although the submanifold is not totally geodesic in general.

Now for almost all choices of~$v$ and~$w$, we can apply our two-dimensional result to this submanifold~$S_{v,w}$.
Issues can arise when boundaries of the simplices are tangent to~$S_{v,w}$ at~$x$, but this is rare.
In such cases~$f$ has to vanish on~$S_{v,w}$ and therefore on all simplices that meet this submanifold near~$x$.
For any simplex containing~$x$ there are such~$w$ and~$v$ (see~\cite{ILS17}), and therefore the claim holds.
We point out that for different pairs $(v,w)$ we get a different submanifold.
\end{proof}

\subsection{Proof of Theorem \ref{thm:thm1}}
\label{sec:sketch} 

We are now ready to prove our main theorem.
We begin with a local version.

Following~\cite{PSUZ16}, we say that a subset $U\subset M$ has a \textit{strictly convex foliation} if there is a strictly convex function $\phi\colon U\to\R$ so that the sets $\{x\in U;\phi(x)\geq c\}$ for all $c>\inf_U\phi$ are compact.

\begin{thm}
\label{thm:supthm}
Let~$M$ be a $C^2$-smooth Riemannian manifold with strictly convex boundary and $\dim(M) \geq 2$.
Suppose that a subset $U \subset M$ has a strictly convex foliation.
Let $W \in C(M; GL(k,\Cset))$ and $f\colon M \to \Cset^k$ be a piecewise constant function.
If $I_W f = 0$ for all geodesics in $U$, then $f|_{U} = 0$. 
\end{thm}

\begin{proof}
The proof is very similar to that of \cite[Theorem 5.3 and Theorem 6.4]{ILS17}, and we only give an outline.

The set~$U$ can be foliated by strictly convex hypersurfaces, and we are interested at the times when the foliation meets a new simplex.
It suffices to prove local injectivity in the neighborhood of a strictly convex boundary point --- a point on the leaf of a foliation --- whenever a new simplex is met. If the point meets only one new simplex, then one can use a sequence of short geodesics that pass the simplex and argue as in Lemma~\ref{keylemma2} to see that $f$ has to vanish in the simplex.
If the point meets more simplices, we are in the setting of Lemma~\ref{lem:kulmat}, and we can conclude that the function vanishes on each new simplex.
\end{proof}

Theorem~\ref{thm:supthm} also has corollaries analogous to \cite[Corollaries 6.5--6.7]{ILS17} but we omit them here.

\begin{proof}[Proof of Theorem~\ref{thm:thm1}]
Under the assumptions there is a global foliation: we may choose $U=M$ in Theorem~\ref{thm:supthm}.
See~\cite[Section 2]{PSUZ16} for details.
The proof is complete.
\end{proof}


\def\bibfont{\footnotesize}
\bibliographystyle{alpha}
\bibliography{sample}

\end{document}